\documentclass[12pt,reqno]{amsart}

\usepackage{amssymb,latexsym,hyperref,tikz}
\usetikzlibrary{calc}

\newcommand{\A}{\mathrm{A}}  
\newcommand{\Aut}{\mathrm{Aut}}
\newcommand{\bbF}{\mathbb{F}} 
\newcommand{\Cay}{\mathrm{Cay}}
\newcommand{\D}{\mathrm{D}}
\newcommand{\E}{\mathrm{E}}
\newcommand{\Nor}{\mathbf{N}}
\newcommand{\PSL}{\mathrm{PSL}}
\newcommand{\Sz}{\mathrm{Sz}}
\newcommand{\V}{\mathrm{V}}
\newcommand{\ZZ}{\mathrm{C}}

\newtheorem{theorem}{Theorem}[section]
\newtheorem{lemma}[theorem]{Lemma}
\newtheorem{proposition}[theorem]{Proposition}
\newtheorem{corollary}[theorem]{Corollary}

\theoremstyle{definition}

\begin{document}

\title[Oriented regular representations of finite simple groups]{Oriented regular representations of out-valency two for finite simple groups}

\author[Verret]{Gabriel Verret}
\address{(Verret) Department of Mathematics\\ University of Auckland\\ PB 92019\\ Auckland\\ New Zealand}
\email{g.verret@auckland.ac.nz}

\author[Xia]{Binzhou Xia}
\address{(Xia) School of Mathematics and Statistics\\University of Melbourne\\Parkville, VIC 3010\\Australia}
\email{binzhoux@unimelb.edu.au}


\begin{abstract}
In this paper, we show that every finite simple group of order at least $5$ admits an oriented regular representation of out-valency $2$.
\end{abstract}

\maketitle

\section{Introduction}

All groups and digraphs in this paper are finite. A \emph{digraph} $\Gamma$ consists of a set of vertices $\V(\Gamma)$ and a set of arcs $\A(\Gamma)$, each arc being an ordered pair of vertices.  A digraph  is \emph{proper} if $(u,v)$ being an arc implies that $(v,u)$ is not an arc. The automorphisms of $\Gamma$ are the permutations of $\V(\Gamma)$ that preserve $\A(\Gamma)$. Under composition, they form the automorphism group $\Aut(\Gamma)$ of $\Gamma$.

Let $G$ be a group and $S\subseteq G$. The \emph{Cayley digraph} $\Cay(G,S)$ on  $G$ with connection set $S$ is the digraph with vertex set $G$ and $(u,v)$ being an arc whenever $vu^{-1}\in S$. Note that $\Cay(G,S)$ is a proper digraph if and only if $S\cap S^{-1}=\emptyset$. Note also that every vertex $u$ in $\Cay(G,S)$ is contained in exactly $|S|$ arcs of the form $(u,v)$. We thus say that $\Cay(G,S)$ has \emph{out-valency} $|S|$.

It is easy to see that $\Aut(\Cay(G,S))$ contains the right regular representation of $G$. If this containment is actually equality, then $\Cay(G,S)$ is called a \emph{digraphical regular representation} (or DRR) of $G$. A DRR that is a proper digraph is called an \emph{oriented regular representation} (or ORR). 

Babai proved that, apart from five small groups, all groups admit a DRR~\cite[Theorem 2.1]{Babai}. He also asked which groups admit ORRs~\cite[Problem 2.7]{Babai}. This was answered by Morris and Spiga~\cite{MorrisSpigaORR1,MorrisSpigaORR2,SpigaORR} who showed that apart from generalised dihedral groups and a small list of exceptions, all groups admit ORRs.

In view of the above, a natural problem is to find ``nice'' DRRs and ORRs, say of ``small'' out-valency. Clearly, only cyclic groups can have DRRs of out-valency $1$, so out-valency $2$ is the smallest interesting case. In this paper, we give the most satisfactory answer to this question in the case of simple groups.

\begin{theorem}\label{ThmMain}
Every finite simple group of order at least $5$ has a ORR of out-valency $2$. 
\end{theorem}

A corollary of Theorem~\ref{ThmMain} is that every nonabelian simple group has a DRR of out-valency $2$. However, the latter conclusion is an immediate consequence of the fact that every nonabelian simple group is generated by an involution and a non-involution (even by an involution and an element of odd prime order, see~\cite[Theorem~1]{King2017}). Indeed, consider a Cayley digraph on a nonabelian simple group with connection set consisting of such a generating pair. This digraph has out-valency $2$, but one out-neighbour of every vertex is also an in-neighbour while the other out-neighbour is not. This implies that fixing a vertex must also fix its out-neighbours and, by connectedness, the whole digraph, and the digraph is a DRR.

Note that Cayley digraphs of out-valency two of simple groups were previously studied in~\cite{FangLuWangXuLowValency}. Another interesting variant of this question would be to consider undirected graphs. In this case, the smallest interesting valency is $3$. The question of which simple groups admit graphical regular representations of valency $3$ has received some attention but is still open~\cite{SpigaCubic,XiaCubic1,XiaCubic2,XiaFangCubic}.

\section{Preliminaries}
\subsection{Generation of finite simple groups}

In this section we present some generation properties of finite simple groups, which will be needed in the proof of Theorem~\ref{ThmMain}. The following result is due to Guralnick and Kantor~\cite[Corollary]{GK2000}.

\begin{theorem}[Guralnick-Kantor]\label{ThmGK}
Every nontrivial element of a finite simple group belongs to a pair of elements generating the group.
\end{theorem}

Note that Theorem~\ref{ThmGK} depends on the classification of finite simple groups.

\begin{corollary}\label{CorGK}
Let $G$ be a finite nonabelian simple group  with an element $x$ of order $3$. Then there exists $y\in G$ such that $|y|\geq4$ and $G=\langle x,y\rangle$.
\end{corollary}

\begin{proof}
By Theorem~\ref{ThmGK} there exists $z\in G$ such that $G=\langle x,z\rangle$. Note that $\langle x,z\rangle=\langle x,xz\rangle$ hence, if either $z$ or $xz$ has order at least $4$, then the conclusion holds (by taking $y=z$ or $y=xz$). We may thus assume that $z$ and $xz$ both have order at most $3$. This implies that $G$ is a quotient of the finitely presented group
\[
\langle x,z\mid x^3,z^m,(xz)^n\rangle
\]
with $m,n\leq 3$. This is the ``ordinary'' $(3,m,n)$ triangle group which is well known to be solvable when $m,n\leq 3$ (see for example~\cite{ConderTriangle}) and therefore so is $G$, which is a contradiction.
\end{proof}

The only nonabelian simple groups with no elements of order $3$ are the Suzuki groups (see~\cite[Page~8, Table~I]{GLS1994}), which we now consider. For a positive integer $m$ and prime number $p$, a prime number $r$ is called a \emph{primitive prime divisor} of $p^m-1$ if $r$ divides $p^m-1$ but does not divide $p^k-1$ for any positive integer $k<m$. By Zsigmondy's theorem~\cite{Zsigmondy1892}, $p^m-1$ has a primitive prime divisor whenever $m\geq3$ and $(p,m)\neq(2,6)$.

\begin{proposition}\label{PropSuzuki}
Let $G=\Sz(q)$ with $q=2^{2n+1}\geq8$ and let $r$ be a primitive prime divisor of $q^4-1$. Then $r\geq 5$, $G$ has an element $y$ of order $r$ and, for each such $y$, there exists $x\in G$ such that $|x|=4$, $|xy|\geq3$ and $G=\langle x,y\rangle$.
\end{proposition}

\begin{proof}
First, recall that $|G|=q^2(q^2+1)(q-1)$  (see~\cite[Page~8, Table~I]{GLS1994}). Since $r$ is a primitive prime divisor of $q^4-1$, it divides $q^4-1$ but not $q^2-1$ and thus must divide $q^2+1$. It follows that $G$ has an element $y$ of order $r$ and that $r\geq 5$. We will now prove that there exists an element $x$ of order $4$ with the required properties, essentially by a somewhat crude counting argument.

We denote by $\E_q$ the elementary abelian group of order $q$ and, for an integer $n\geq2$, by $\ZZ_n$ the cyclic group of order $n$ and  $\D_{2n}$ the dihedral group of order $2n$.   

Up to conjugation, the maximal subgroups of $G$ are the following (see for instance~\cite[Table~8.16]{BHR2013}):
\begin{itemize}
\item $(\E_q.\E_q)\rtimes\ZZ_{q-1}$,
\item $\D_{2(q-1)}$,
\item $\ZZ_{q+\sqrt{2q}+1}\rtimes\ZZ_4$,
\item $\ZZ_{q-\sqrt{2q}+1}\rtimes\ZZ_4$,
\item $\Sz(q_0)$, where $q_0=q^{1/d}>2$ for some prime divisor $d$ of $2n+1$.
\end{itemize}
Recall that $r$ is odd,  does not divide $q-1$ nor $q_0^4-1$ and thus does not divide its factor $(q_0^2+1)(q_0-1)$. This implies that $r$ does not divide $|\Sz(q_0)|=q_0^2(q_0^2+1)(q_0-1)$.  It follows that a maximal subgroup $M$ of $G$ containing $y$ must be of the form $\ZZ_{q\pm \sqrt{2q}+1}\rtimes\ZZ_4$. Since every subgroup of a cyclic group is characteristic, $\langle y\rangle$ is normal in $M$ and thus $M$ is the only maximal subgroup of $G$ containing $y$  (for otherwise $\langle y\rangle$ would be normal in another maximal subgroup $N$ of $G$ and thus normal in $\langle M,N\rangle=G$).

Let $Q$ be a Sylow $2$-subgroup of $G$. Then $Q=\E_q.\E_q$ and $|\Nor_G(Q)|=(\E_q.\E_q)\rtimes\ZZ_{q-1}$. Hence the number $n$ of Sylow $2$-subgroups of $G$ is
\[
n=\frac{|G|}{|\Nor_G(Q)|}=\frac{q^2(q^2+1)(q-1)}{q^2(q-1)}=q^2+1.
\]
Let $n_2$ and $n_4$ denote the numbers of elements of order $2$ and $4$, respectively, in $G$. According to~\cite[Lemma~3.2]{FP1999}, there are $q-1$ involutions and $q^2-q$ elements of order $4$ in $Q$, and different conjugates of $Q$ have trivial intersection.  Then
\[
n_2=n(q-1)=(q^2+1)(q-1)
\]
and
\[
n_4=n(q^2-q)=(q^2+1)(q^2-q).
\]

Let
\[
I=\{g\in G:|gy|\leq2\}
\]
and
\[
J=\{g\in G: \langle g,y\rangle\neq G\}.
\]
Then $|I|=n_2+1$ and, since $M$ is the unique maximal subgroup of $G$ containing $y$,  $|J|\leq|M|$. Since
\begin{align*}
|I|+|J|&\leq n_2+|M|+1\\
&=(q^2+1)(q-1)+4(q\pm \sqrt{2q}+1)+1\\
&\leq(q^2+1)(q-1)+4(q+\sqrt{2q}+1)+1\\
&<(q^2+1)(q^2-q)=n_4,
\end{align*}
it follows that there exists $x\in G$ with $|x|=4$ and $x\notin I\cup J$, as required.
\end{proof}

\subsection{Constructing ORRs of out-valency $2$}

\begin{lemma}\label{Lem3-Cycle}
Let $G=\langle x,y\rangle$. If $|x|=3$ and $|y|\geq4$, then $\Cay(G,\{x,y\})$ is an ORR, unless $|y|=6$ and $x=y^4$, and $G\cong\ZZ_6$.
\end{lemma}

\begin{proof}
Let $\Gamma=\Cay(G,\{x,y\})$ and let $A=\Aut(\Gamma)$. Note that $\Gamma$ is a strongly connected proper digraph. The following diagram shows all the directed paths of length at most $3$ in $\Gamma$  starting at $1$.
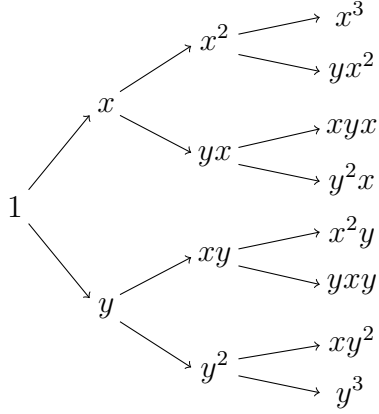
\begin{center}
\begin{tikzpicture}[scale=0.45]
\draw [->] (1.2,0.5)--(3,2.7); \draw [->] (1.2,-0.5)--(3,-2.7);
\draw [->] (3.9,3.4)--(6,4.75); \draw [->] (3.9,2.7)--(6,1.6); \draw [->] (3.9,-2.6)--(6,-1.5); \draw [->] (3.9,-3.4)--(6,-4.75);
\draw [->] (7.4,5.1)--(9.8,5.6); \draw [->] (7.4,4.5)--(9.8,4); \draw [->] (7.4,1.75)--(9.8,2.3); \draw [->] (7.4,1.25)--(9.8,0.75);
\draw [->] (7.4,-5)--(9.8,-5.5); \draw [->] (7.4,-4.5)--(9.8,-4.1); \draw [->] (7.4,-1.75)--(9.8,-2.3); \draw [->] (7.4,-1.25)--(9.8,-0.75);
\node at (0.8,0) {$1$};
\node at (3.5,3) {$x$}; \node at (3.5,-3) {$y$};
\node at (6.7,5) {$x^2$}; \node at (6.7,1.5) {$yx$}; \node at (6.7,-4.75) {$y^2$}; \node at (6.7,-1.5) {$xy$};
\node at (10.7,5.7) {$x^3$}; \node at (10.75,4.1) {$yx^2$}; \node at (10.75,2.3) {$xyx$}; \node at (10.75,0.8) {$y^2x$};
\node at (10.7,-5.5) {$y^3$}; \node at (10.75,-4) {$xy^2$}; \node at (10.75,-2.3) {$yxy$}; \node at (10.75,-0.75) {$x^2y$};
\end{tikzpicture}
\end{center}
Since $|y|\geq 4$, we have $y^3\neq 1$ and $y\neq x^{-2}$. Moreover, if $y^2=x^{-1}$, then $|y|=6$ and thus $x=y^4$ and the result holds. We thus assume this is not the case. Since $x^3=1$, this implies that $(1,x,x^2,x^3)$ is the only directed cycle of length $3$ starting at $1$. This implies that the stabiliser $A_1$ of the vertex $1$ also fixes $x$. As $1$ only has one out-neighbour other than $x$, it must also be fixed. By vertex-transitivity, we find that fixing a vertex fixes its out neighbours and, using connectedness, we conclude that $A_1=1$ and thus $\Gamma$ is an ORR.
\end{proof}

\begin{lemma}\label{Lem4-Cycle}
Let $G=\langle x,y\rangle$. If $|x|=4$, $|y|\geq5$ and $|xy|\geq3$, then $\Cay(G,\{x,y\})$ is an ORR, unless $|y|=12$ and $x=y^9$, and $G\cong\ZZ_{12}$.
\end{lemma}

\begin{proof}
Let $\Gamma=\Cay(G,\{x,y\})$ and let $A=\Aut(\Gamma)$. Note that $\Gamma$ is a strongly connected proper digraph. The following diagram shows all the directed paths of length at most $4$ in $\Gamma$  starting at $1$.

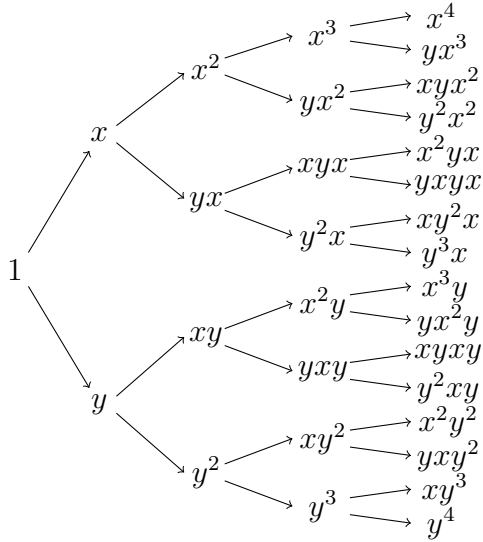
\begin{center}
\begin{tikzpicture}[scale=0.45]
\draw [->] (1.2,0.5)--(3,3.5);\draw [->] (1.2,-0.5)--(3,-3.5);

\draw [->] (3.8,4.25)--(5.8,5.8);\draw [->] (3.8,3.75)--(5.8,2.1);\draw [->] (3.8,-3.75)--(5.8,-2);\draw [->] (3.8,-4.25)--(5.8,-6);

\draw [->] (7,6.25)--(9,6.9); \draw [->] (7,5.75)--(9,5);\draw [->] (7,2.25)--(9,3);\draw [->] (7,1.75)--(9,1);
\draw [->] (7,-6.25)--(9,-7); \draw [->] (7,-5.75)--(9,-5);\draw [->] (7,-2.25)--(9,-3);\draw [->] (7,-1.75)--(9,-1);

\draw [->] (10.7,7.2)--(12.5,7.5);\draw [->] (10.7,6.75)--(12.5,6.5);\draw [->] (10.7,5.25)--(12.5,5.5);\draw [->] (10.7,4.75)--(12.5,4.5);
\draw [->] (10.7,3.25)--(12.5,3.5);\draw [->] (10.7,2.75)--(12.5,2.5);\draw [->] (10.7,1.25)--(12.5,1.5);\draw [->] (10.7,0.75)--(12.5,0.5);
\draw [->] (10.7,-7.25)--(12.5,-7.5);\draw [->] (10.7,-6.75)--(12.5,-6.5);\draw [->] (10.7,-5.25)--(12.5,-5.5);\draw [->] (10.7,-4.75)--(12.5,-4.5);
\draw [->] (10.7,-3.25)--(12.5,-3.5);\draw [->] (10.7,-2.75)--(12.5,-2.5);\draw [->] (10.7,-1.25)--(12.5,-1.5);\draw [->] (10.7,-0.75)--(12.5,-0.5);

\node at (0.8,0) {$1$};

\node at (3.3,4) {$x$};\node at (3.3,-4) {$y$};

\node at (6.45,6) {$x^2$};\node at (6.45,2) {$yx$};\node at (6.45,-6) {$y^2$};\node at (6.45,-2) {$xy$};

\node at (9.9,7) {$x^3$};\node at (9.9,5) {$yx^2$};\node at (9.9,3) {$xyx$};\node at (9.9,1) {$y^2x$};
\node at (9.9,-7) {$y^3$};\node at (9.9,-5) {$xy^2$};\node at (9.9,-3) {$yxy$};\node at (9.9,-1) {$x^2y$};

\node at (13.4,7.5) {$x^4$};\node at (13.5,6.5) {$yx^3$};\node at (13.6,5.5) {$xyx^2$};\node at (13.6,4.5) {$y^2x^2$};\node at (13.6,3.5) {$x^2yx$};\node at (13.6,2.5) {$yxyx$};\node at (13.6,1.5) {$xy^2x$};\node at (13.5,0.5) {$y^3x$};
\node at (13.4,-7.5) {$y^4$};\node at (13.5,-6.5) {$xy^3$};\node at (13.6,-5.5) {$yxy^2$};\node at (13.6,-4.5) {$x^2y^2$};\node at (13.6,-3.5) {$y^2xy$};\node at (13.6,-2.5) {$xyxy$};\node at (13.6,-1.5) {$yx^2y$};\node at (13.5,-0.5) {$x^3y$};
\end{tikzpicture}
\end{center}
Since $|y|\geq 5$, we have $y^4\neq 1$,  $y\neq x^{-3}$ and $y^2\neq x^{-2}$. Similarly, $|xy|\geq3$ implies that $(xy)^2\neq 1\neq (yx)^2$. Moreover, if $y^3=x^{-1}$, then $|y|=12$ and thus $x=y^9$ and the result holds. We thus assume this is not the case. Since $x^4=1$, this implies that $(1,x,x^2,x^3,x^4)$ is the only directed cycle of length $4$ starting at $1$ and, as in the previous lemma, $\Gamma$ is an ORR.
\end{proof}

\section{Proof of Theorem~\ref{ThmMain}}
Let $G$ be a finite simple group with $|G|\geq5$. We first suppose that $G=\bbF_p^+$ for some prime $p\geq5$. Let $x,y\in\bbF_p\setminus\{0\}$ such that $x\neq\pm y$ and let $\Gamma=\Cay(G,\{x,y\})$. Note that $\Gamma$ is a proper digraph of out-valency $2$. By~\cite[Proposition~1.3~and~Example~2.2]{Xu1998}, $\Gamma$ is an ORR if and only if  the only solution to 
\begin{equation}\label{EqAut(G,S)}
\{\lambda x,\lambda y\}=\{x,y\}
\end{equation}
with $\lambda\in\bbF_p^\times$ is $\lambda=1$. Suppose otherwise, that is~\eqref{EqAut(G,S)} holds with $\lambda\neq 1$. This implies that $\lambda x=y$ and $\lambda y=x$, which yields that
\[
\lambda x^2=(\lambda x)x=y(\lambda y)=\lambda y^2,
\]
and hence $x^2=y^2$, contradicting $x\neq\pm y$. Thus we conclude that $\Gamma$ is an ORR, as required.

We may now assume that $G$ is nonabelian. If $G$ has an element $x$ of order $3$ then, by Corollary~\ref{CorGK} there exists $y\in G$ such that $|y|\geq4$ and $G=\langle x,y\rangle$. By Lemma~\ref{Lem3-Cycle},  $\Cay(G,\{x,y\})$ is an ORR. We may thus assume that $G$ does not have an element of order $3$ and thus  $G=\Sz(q)$ for some $q=2^{2n+1}\geq8$. Let $r$ be a primitive prime divisor of $q^4-1$. By Proposition~\ref{PropSuzuki}, $G$ contains elements $x$ and $y$ such that $|x|=4$, $|y|=r\geq5$, $|xy|\geq3$ and $G=\langle x,y\rangle$.  By Lemma~\ref{Lem4-Cycle}, $\Cay(G,\{x,y\})$ is an ORR.
\qed

\section*{Acknowledgements}
The authors are grateful to the N.Z. Marsden Fund which helped support (via grant UOA1824) the second author's visit to the University of Auckland in 2019.

\end{document}